\documentclass{amsart}
\usepackage{amssymb,amsmath}

\newtheorem{theorem}{Theorem}
\newtheorem{corollary}{Corollary}

\begin{document}

\title{The part-frequency matrices of a partition}
\author{William J. Keith}

\begin{abstract}

A new combinatorial object is introduced, the part-frequency matrix sequence of a partition, which is elementary to describe and is naturally motivated by Glaisher's bijection. We prove results that suggest surprising usefulness for such a simple tool, including the existence of a related statistic that realizes every possible Ramanujan-type congruence for the partition function. To further exhibit its research utility, we give an easy generalization of a theorem of Andrews, Dixit and Yee \cite{ADY} on the mock theta functions.  Throughout, we state a number of observations and questions that can motivate an array of investigations. 

\end{abstract}

\maketitle

\section{Acknowledgements}

This note is an expansion of a short talk given at ``Algebraic Combinatorics and Applications,'' the first annual Kliakhandler Conference at Michigan Technological University.  The author would like to thank Igor Kliakhandler for his generous support of mathematics at Michigan Tech and Vladimir Tonchev for organizing the 2015 Conference.

The anonymous referee is also appreciated for a sharp eye contributing clarity to the presentation in several places.

\section{Introduction}

A nonincreasing sequence of positive integers $(\lambda = (\lambda_1, \lambda_2, \dots , \lambda_i)$ such that $\sum \lambda_j = n$ is a partition of $n$, denoted $\lambda \vdash n$.  We say that $n$ is the weight of the partition, or write $\vert \lambda \vert = n$.  Denote the number of partitions of $n$ by $p(n)$.  Their generating function is $$P(n) := \sum_{n=0}^\infty p(n) q^n = \prod_{i=1}^\infty \frac{1}{1-q^i}.$$

One of the first and most useful theorems learned by a student of partition theory is Euler's result that the number of partitions of $n$ into odd parts (e.g., $3+1$ or $1+1+1+1$ for $n=4$) is equal to the number of partitions of $n$ into distinct parts (4 or $3+1$).  This enumeration theorem, interpreted in generating functions, gives the $q$-series identity $$ \prod_{k=1}^\infty \frac{1}{1-q^{2k-1}} = \prod_{k=1}^\infty (1+q^k) .$$  Glaisher's bijection gives a direct mapping between the two sets: if in a partition into odd parts the part $i$ appears $a_{i,0} 2^0 + a_{i,1} 2^1 + \dots $ times, $a_{i,\ell} \in \{0,1\}$, then construct the distinct parts $a_{i,\ell} \, i \cdot 2^\ell$.  The reverse map collects parts with common odd factor $i$.  The identity and Glaisher's map both generalize in the obvious way to partitions into parts not divisible by $m$ and partitions in which parts appear fewer than $m$ times.

Observe that for each $m$, Glaisher's map can be extended to an involution on all partitions: construct a list of matrices $M_j$ indexed by the numbers $j$ not divisible by $m$.  Enumerate columns and rows starting with 0.  If the part $N = j m^k$ appears $a_{N,0} m^0 + a_{N,1} m^1 + a_{N,2} m^2 + \dots$ times, then assign row $k$ of matrix $M_j$ to be $$\begin{matrix} a_{k,1} & a_{k,2} & a_{k,3} & \dots \end{matrix} \quad := \quad \begin{matrix} a_{N,0} & a_{N,1} & a_{N,2} & \dots \end{matrix}.$$   For example, if $m=2$, the partition $(20,5,5,4,2,2,1,1,1,1,1)$ of 43 would be depicted (with all other entries 0):

\begin{center} \begin{tabular}{ccc} $\begin{array}{c|ccc} 1 & & & \\ \hline & 1 & 0 & 1 \\ & 0 & 1 & 0 \\ & 1 & 0 & 0 \end{array}$ & $\begin{array}{c|ccc} 3 & & & \\ \hline & 0 & 0 & 0 \\ & 0 & 0 & 0 \\ & 0 & 0 & 0 \end{array}$ & $\begin{array}{c|ccc} 5 & & & \\ \hline & 0 & 1 & 0 \\ & 0 & 0 & 0 \\ & 1 & 0 & 0 \end{array}$ . \end{tabular} \end{center}

Now it is easy to see that Glaisher's map is simply transposition of this sequence of matrices, restricted to partitions into odd or distinct parts, where either the first row or the first column are the only nonzero entries.

Once these matrices are built, a wide array of transformations suggest themselves.  We will consider several in this paper.  The crucial observation is that \emph{any} bijection which preserves the sums in the SW-NE antidiagonals of the matrices will preserve the weights of the partitions under consideration.  For an object so easily related to such a classical theorem in partitions, this simple idea seems to have a surprising range of possible applications.

In section \ref{OrbitSection} we consider the action of rotating antidiagonals through their length, and in doing so find several theorems including new statistics that realize any of the known congruences for the partition function, including Ramanujan's congruences mod 5, 7, and 11.  In section \ref{MockSection} we generalize a theorem from Andrews', Dixit's and Yee's paper \cite{ADY} related to $q$-series identities concerning partitions into odd and even parts which arises from third order mock theta functions.  Throughout, we mention questions of possible research interest which arise in connection with these ideas.

\section{A congruence statistic: the orbit size}\label{OrbitSection}

A second famous result in partition theory is the set of congruences for the partition function observed by Ramanujan, the first examples of which are 

\begin{align*} p(5n+4) &\equiv 0 \pmod{5} \\p(7n+5) &\equiv 0 \pmod{7} \\ p(11n+6) &\equiv 0 \pmod{11}.
\end{align*}

In fact for any $\alpha, \beta, \gamma \geq 0$, there is a $\delta$ such that $$p(5^\alpha 7^{\lfloor \frac{\beta+1}{2}\rfloor} 11^\gamma n + \delta) \equiv 0 \pmod{5^\alpha 7^\beta 11^\gamma}.$$

The first two congruences can be combinatorially realized by Dyson's \emph{rank} of a partition, which is simply the largest part minus the number of parts.  The number of partitions $\lambda$ of $5n+4$ is divided into five equal classes of size $\frac{p(5n+4)}{5}$ by the residue mod 5 of the rank of $\lambda$, and likewise seven classes by the residue of $7n+6$.

A more important statistic, the \emph{crank}, realizes all the congruences listed, and all of the other known congruences for the partition function, although the classes so constructed divide $p(An+B)$ into classes of possibly unequal sizes divisible by $C$ for the various progressions $p(An+B) \equiv 0 \pmod{C}$.  The crank has proven to be of great use in partition theory (see for instance \cite{Crank1}, \cite{Crank2}, \cite{Crank3} for just a few papers employing and illuminating this statistic).  Although combinatorially defined, proofs for properties of the crank typically require and illuminate deep and beautiful properties of modular forms.

We now claim that associated to the sequence of part-frequency matrices of partitions is a class of natural combinatorial statistics which realizes not only every known congruence but indeed every possible congruence for the partition function.

To construct this statistic we turn to a natural transformation on the part-frequency matrix: rotating the antidiagonals.  To be precise, given modulus $m$ and describing any part size as $j m^k$ with $m \, \nmid j$, define the map $\rho$ on a partition which, operating on the part-frequency matrices $(a_{i,\ell})$, yields $$\left(\rho(M_j)\right)_{i,\ell} = a_{i-1,\ell+1} \text{ for } i>1 \, \text{ and } \left(\rho(M_j)\right)_{1,\ell} =  a_{\ell,1}.$$

That is, we move each entry down and left 1 place, and the lowest entry in an antidiagonal is moved to the upper end.

When this action is applied to partitions in which parts appear fewer than $m$ times, i.e. nonzero entries only appear in the first column, all nonzero entries rotate to the first row.  Thus, this map still contains Glaisher's bijection as a special case.

Our interest now is in the size of the orbits among partitions produced by this action.  For example, let $m=2$ and consider the partition $(20,5,5,4,2,2,1,1,1,1,1)$ illustrated in the introduction.  We see that $M_1$, $M_3$ and $M_j$ for $j\geq 7$ are unchanged by this operation.  In the case of $M_1$, all antidiagonals are ``full'' in that if any entry in the antidiagonal is 1, all are, while the others are empty.  However, $M_5$ is altered by this operation, with an orbit of length 6.  Its images are, in order: 

\phantom{.}

\begin{tabular}{ccc}  $\begin{array}{c|ccc} 5 & & & \\ \hline & 0 & 1 & 0 \\ & 0 & 0 & 0 \\ & 1 & 0 & 0 \end{array}$ & $\begin{array}{c|ccc} 5 & & & \\ \hline & 0 & 0 & 1 \\ & 1 & 0 & 0 \\ & 0 & 0 & 0 \end{array}$ & $\begin{array}{c|ccc} 5 & & & \\ \hline & 0 & 1 & 0 \\ & 0 & 1 & 0 \\ & 0 & 0 & 0 \end{array}$ \\ \quad \quad $\begin{array}{c|ccc} 5 & & & \\ \hline & 0 & 0 & 0 \\ & 1 & 0 & 0 \\ & 1 & 0 & 0 \end{array}$ & \quad \quad $\begin{array}{c|ccc} 5 & & & \\ \hline & 0 & 1 & 1 \\ & 0 & 0 & 0 \\ & 0 & 0 & 0 \end{array}$ & \quad \quad $\begin{array}{c|ccc} 5 & & & \\ \hline & 0 & 0 & 0 \\ & 1 & 1 & 0 \\ & 0 & 0 & 0 \end{array}$ \end{tabular}

\phantom{.}

\noindent Thus, overall, the partition $(20,5,5,4,2,2,1,1,1,1,1)$ is part of an orbit of size 6.

\phantom{.}

One immediately observes that information about the sizes of the orbits of this action will give us information about $p(n)$.  In particular, denote by $p_{\rho,m}(n,k)$ the number of partitions of $n$ that lie in orbits of size $k$ under the action of $\rho$ with modulus $m$, and by $o_{\rho,m}(n,k)$ the number of orbits of size $k$.  We then have that

\begin{theorem}\label{CongruencePtns} Given any partition congruence $p(An+B) \equiv 0 \pmod{C}$, if $B < A$ and $A \vert m$, then $p_{\rho,m}(An+B,k) \equiv 0 \pmod{C}$ for all $k$.
\end{theorem}

\begin{theorem}\label{CongruenceOrbits} Under the previous hypotheses, $o_{\rho,m}(An+B,k) \equiv 0 \pmod{C}$ for all $k$ also.
\end{theorem}

Thus, \emph{every} linear congruence for the partition function can be realized by the orbit statistic for at least $m = A$.  For example, the orbit sizes for modulus 5, 7, and 11 realize Ramanujan's congruences, that $p(mn+k) \equiv 0 \pmod{m}$ for $(m,k) \in \{(5,4), (7,5), (11,6)\}$.  It seems interesting, and hopefully useful, that a map which is a small modification of Glaisher's old bijection should have such close connections to the overarching structure of partition congruences now known to exist.

First we prove Theorem \ref{CongruencePtns}.  We will employ the standard notation $$\prod_{i=1}^\infty 1-q^i =: (q;q)_\infty.$$

\noindent \emph{Proof of Theorem \ref{CongruencePtns}.}

Observe partitions into parts not divisible by $m$, appearing fewer than $m$ times.  These are those partitions whose part-frequency matrices are nonzero in no entry other than the upper-leftmost corner.  Refer to such a partition as an \emph{upper-left filling}.  Denote the number of these by $p_{m,m}(n)$ and their generating function by $P_{m,m}(q)$.  This function is 

\begin{multline}P_{m,m}(q) := \sum_{n=0}^\infty p_{m,m} (n) q^n := \prod_{{k=1} \atop {m \nmid k}}^\infty 1+q^k+\dots+q^{(m-1)k} = \prod_{{k=1} \atop {m \nmid k}}^\infty \frac{1-q^{mk}}{1-q^k} \\
= \prod_{{k=1} \atop {m \nmid k}}^\infty \frac{1-q^{mk}}{1-q^k} \prod_{j=1}^\infty \frac{1-q^{mj}}{1-q^{mj}} = \frac{(q^m;q^m)_\infty^2}{(q;q)_\infty (q^{m^2};q^{m^2})_\infty}.
\end{multline}

Thus

$$P_{m,m}(q) = \prod_{i=1}^\infty \frac{1}{1-q^i} \frac{(q^m;q^n)_\infty^2}{(q^{m^2};q^{m^2})_\infty} = \left( \sum_{n=0}^\infty p(n) q^n \right) \left( \sum_{i=0}^\infty c(i) q^{mi} \right)$$

\noindent where the $c(i)$ are the integral coefficients of the power-series expansion of the second factor in the middle term.  Note that the powers of $q$ in this function with nonzero coefficients are all multiples of $m$.  This gives us a recurrence for $p_{m,m} (n)$ in terms of values of the original partition function $p(n-\ell m)$ for various nonnegative integral $\ell$: $$p_{m,m}(n) = p(n)c(0) + p(n-m)c(1) + p(n-2m)c(2) + \dots .$$

If all terms in the progression are divisible by $C$, then $p_{m,m}(n)$ will be as well.  If $n \equiv B \pmod{A}$, $A \vert m$, and $p(Ak+B) \equiv 0 \pmod{C}$ for all $k$, then it will hold that $p(n-km) \equiv 0 \pmod{C}$.

\phantom{.}

\noindent \textbf{Remark 1:} The upper left fillings are precisely the fixed points of Glaisher's bijection.  This makes them a very natural object of interest, and yet it appears that only one class of them has been well-studied: the $m=2$ case of partitions of $n$ into parts both odd and distinct, the number of which are well-known to be of equal parity with partitions of $n$, as they are in bijection with the self-conjugate partitions of $n$.  The generating functions of upper-left fillings are simple $\eta$-quotients, which by work of Treener \cite{Treneer} will possess their own congruences, but of perhaps greater interest is that many of these, and variants in which we permit the divisibility and frequency moduli to be different, appear in connection with McKay-Thompson series \cite{McKay}.  Investigation of these objects thus seems like it could be interesting in its own right.

\phantom{.}

Now observe that we may write partitions in which the orbit size is any stipulated value as an upper left filling plus the presence of various specified sets of fillings of other diagonals.

For instance, suppose $m=5$.  Then any partition consisting of an upper-left filling of some kind, 1s in the entries $a_{0,1}$ and $a_{1,0}$ for $M_1$, and a 3 in the entry $a_{0,1}$ for $M_4$, representing 1s appearing 5 times (plus up to 4 more appearances), a single 5, and a 4 appearing 15 times (plus up to 4 more appearances), and no other entries, will have orbit size 2.  The number of such partitions is just the number of upper-left filling partitions of $n-70$.

This gives us a recurrence for $p_{\rho,m}(n,k)$, the number of partitions of $n$ of orbit size $k$, in terms of the partitions that are upper-left fillings.  Obviously this would be an exceedingly complicated recurrence, but we only require that each of the terms in the recurrence is of the form $p_{m,m} (n-\ell m)$ for some nonnegative integral $\ell $.  Since in the arithmetic progressions $An+B$ we have established that $p_{m,m} (An+B)$ is 0 mod $C$, the numbers $p_{\rho,m}(An+B,k)$ will themselves be 0 mod $C$.

\hfill $\Box$

\phantom{.}

\noindent \emph{Proof of Theorem \ref{CongruenceOrbits}.}  Each orbit is an equivalence class of antidiagonals under the action of rotation.  For instance, one orbit of size 6 is the equivalence class represented by the six $M_5$ matrices in the opening example partition, and the fixed matrices making up the rest of the sequence.  For such an orbit, define the \emph{weight} of the orbit as the sum of the entries in the antidiagonals outside the upper left corner, multiplied by the part-size powers of $m$ and the frequency power of $m$ associated to each element, i.e. the sum $\sum_{i,j} a_{i,j} m^{i j}.$  For the example partition, the weight is 52 -- only the 1 in the upper left corner of $M_1$ is part of the upper-left filling.

For a given orbit, the weight in the antidiagonals outside of the upper left corner is some constant multiple of $m$, say $Km$.  Then for each equivalence class (under rotation) of arrangements of entries in the antidiagonals of length greater than 1, there is one orbit per upper left filling of $n-Km$.  But if $n \equiv B \pmod{A}$, then $n-Km \equiv B \pmod{A}$ since $A \vert m$, and so the upper left fillings of $n-Km$, being defined by some recurrence in $p(n-Km-\ell m)$, also possess the property that their number is divisible by $C$.  Thus, the number of orbits is divisible by $C$. 

\hfill $\Box$

\phantom{.}

\noindent \textbf{Remark 2:} A reader may prudently object that if some sequence $f(n)$, $n \geq 0$  has the property $f(Ak+B) \equiv 0 \pmod{C}$, and a second power series $F(q)$ has integral coefficients, and we consider the coefficients of the generating function $$\sum_{n=0}^\infty \phi(n) q^n = \left( \sum_{n=0}^\infty f(n) q^n \right) F(q^m)$$ \noindent for some $m$ such that $A \vert m$, the coefficients $\phi(n)$ will obviously share the congruence $\phi(Ak+B) \equiv 0 \pmod{C}$ by a similar recurrence argument.  Of course this is true; we are mathematically interested in such a relation when, (a), the coefficients count an easily describable combinatorial object, and, (b), the function $F(q^m)$ arises in a natural way from a simple action on the elements counted by $f(n)$, which finally (c) poses some hope of obtaining further research utility in the future.

\phantom{.}

The first example of the theorems arises from the first Ramanujan congruence $p(5n+4) \equiv 0 \pmod{5}$.  Modulo 5, the number of orbits of given sizes are:

\begin{center}\begin{tabular}{|c|c|c|c|c|}
\hline $5n+4$ & 1 & 2 & 3 & 6 \\
\hline 4 & 5 & 0 & 0 & 0 \\
\hline 9 & 20 & 5 & 0 & 0 \\
\hline 14 & 75 & 30 & 0 & 0 \\
\hline 19 & 220 & 135 & 0 & 0 \\
\hline 24 & 605 & 485 & 0 & 0 \\
\hline 29 & 1480 & 1535 & 5 & 0 \\
\hline 34 & 3470 & 4375 & 20 & 5 \\
\hline 39 & 7620 & 11580 & 75 & 30 \\
\hline
\end{tabular}
\end{center}

(An orbit of size 6 can arise if there are diagonals with period 2 and 3; this will happen sooner than an orbit of size 4.)

It would have been most interesting if the rank or the crank, restricted to classes of partitions with equal orbit sizes, also divided these classes into equinumerous sets; this is not the case, and so there arises the question

\phantom{.}

\noindent \textbf{Question 1.} Does there exist a simple combinatorial statistic on upper left fillings, partitions of orbit size 1, or the orbits of size 1, that divides these sets into classes of equal size, or classes of size divisible by $m$?

\phantom{.}

If this property could be proved in an elementary fashion, we would have a completely elementary proof of the congruences the statistic supported -- something which, to date, is lacking for Ramanujan's congruences.

\phantom{.}

\noindent \textbf{Remark 3:} If one calculates the Andrews-Garvan Dyson crank restricted to partitions in which parts are not divisible by $m$ and in which parts appear fewer than $m$ times, one notices a significant deficiency in the number of partitions with crank 0 mod $m$.  The reason for this is when there are no 1s in the partition, the crank is simply the size of the largest part, and under these restrictions this group never contributes a crank 0 mod $m$.  A modification will have to figure out a way to work around this difficulty.

\phantom{.}

\noindent \textbf{Question 2.} Recent work by Breuer, Eichhorn and Kronholm \cite{Brandt} gives a geometric statistic they label a ``supercrank'' that realizes all congruences for partitions into three parts via the transformations of a cubic fundamental domain in the cone that describes these partitions (with work in progress on partitions into general fixed numbers of parts).  It is probably unlikely that these statistics are intimately related, but any connection between the two families would be fascinating.

\subsection{Generating functions for orbit sizes}

A natural question that arises is the generating function for each orbit size.

Observe that an orbit of size 1 appears only when an antidiagonal, if occupied by anything other than 0, must be filled with copies of the same digit.  Suppose an antidiagonal in $M_j$ is of length $k$.  Then its entries are the elements $a_{i,\ell}$ with $i+\ell=k-1$, so each element represents a contribution of parts totaling $j a_{i,\ell} m^{k-1}$ to the partition.  Thus, a fixed diagonal can be considered as contributing $k$ copies of $j m^{k-1}$ some number of times $a_{i,\ell}$ ranging over $0 \leq a_{i,\ell} < m$.

Thus, if the matrices $M_j$ are modulus $m$, the orbits of size 1 can be described as partitions of $n$ into parts $k \cdot j \cdot m^{k-1}$ appearing fewer than $m$ times:

\begin{multline*}
P_{m}^{(1)} (q) = \prod_{{j: m \nmid j} \atop {k \geq 1}} \left(1+q^{jkm^{k-1}}+q^{2jkm^{k-1}}+ \cdots + q^{(m-1)jkm^{k-1}} \right) \\
= \prod_{{j: m \nmid j} \atop {k \geq 1}} \frac{1-q^{jkm^k}}{1-q^{jkm^{k-1}}} = \prod_{{j: m \nmid j} \atop {k \geq 1}} \frac{1-q^{jkm^k}}{1-q^{jkm^{k-1}}} \cdot \prod_{b,k \geq 1} \frac{1-q^{(bm)km^{k-1}}}{1-q^{(bm)km^{k-1}}} \\
= \prod_{k \geq 1} \frac{\left(q^{km^k};q^{km^k}\right)_{\infty}^2}{\left(q^{km^{k-1}};q^{km^{k-1}}\right)_{\infty}} \cdot \frac{1}{\left(q^{km^{k+1}};q^{km^{k+1}}\right)_{\infty}} \\
= \frac{1}{(q;q)_{\infty}} \prod_{k \geq 1}  \frac{\left(q^{km^k};q^{km^k}\right)_{\infty}^2}{\left(q^{(k+1)m^{k}};q^{(k+1)m^{k}}\right)_{\infty} \left(q^{km^{k+1}};q^{km^{k+1}}\right)_{\infty}}
\end{multline*}

\phantom{.}

\noindent \textbf{Question 3.} This is a somewhat curious generating function.  It is an infinite $\eta$-quotient, most of the factors of which do not cancel. Finite $\eta$-quotients are weakly holomorphic modular forms, but does this fit in any of the usual classes of generating functions?  Since the number of partitions of orbit size 1 is of interest in possible elementary proofs for partition congruences, what can be said -- whether combinatorially, analytically or asymptotically -- about its coefficients?

\phantom{.}

Orbits of size 2 will will contain either a fixed point as in the orbits of size 1, or fillings of every other entry in an antidiagonal with the same entry.  Let the \emph{population} of an antidiagonal denote the number of its nonzero entries.  Every odd population $k$ thus arises in three ways: either as the filled antidiagonal of length $k$, or half of the antidiagonal of length $2k$, interleaved with zeroes.  Those cases in which both halves appear cover the case of even population $k$ filling the antidiagonal of length $k$, so we need add only two more varieties of even population $k$ to cover the case of the two interleaved halves of the antidiagonal of length $4k$.  Finally, we subtract $P_m^{(1)}$ to account for the count of partitions where any occupied halves happen to match.

Letting $a_k = 3$ if $k$ is odd, and $a_k = 2$ if $n$ is even, we have by logic similar to the previous function,

$$P_m^{(2)} (q) = \prod_{k \geq 1} \left( \frac{\left(q^{km^k};q^{km^k}\right)_{\infty}^2}{\left(q^{km^{k-1}};q^{km^{k-1}}\right)_{\infty}} \cdot \frac{1}{\left(q^{km^{k+1}};q^{km^{k+1}}\right)_{\infty}} \right)^{a_k} - P_m^{(1)}(q)$$

The logic is easily extended, noting that for composite orbit sizes there will be an inclusion-exclusion alternating sum of orbit sizes dividing the desired orbit.




\subsection{Orbits of size not divisible by a given modulus}

One may wonder if this construction could yield any useful information on the parity or 3-arity of the partition numbers.  Unfortunately, it seems rather weak for that purpose.  After all, it is already well known that the partition function $p(n)$ is congruent modulo 2 to the number of partitions of $n$ into distinct odd parts, which is just an upper left filling; analysis of additional orbit sizes could hardly be expected to yield better information.

It is easy to observe that $p(n)$ will be congruent modulo 2 or 3 to the number of partitions in orbits of size not divisible by 2 or 3 respectively.  Suppose that we set out to construct such partitions for modulus 2.  Consider the part-frequency matrices for $m=2$.

In each antidiagonal of length $k$, we desire that the period of the antidiagonal under the rotation action be odd.  Let $k = t \cdot 2^{r(k)}$ where $t$ is the largest odd divisor of $k$.  Then we allow orbits of size $t$ or any divisor of $t$, so we build the antidiagonal of length $k$ by choosing entries in periods of size $t$; there will be $2^{r(k)}$ entries in each such period.  We may choose these freely; different choices will yield different partitions.  The $2^{r(k)}$ entries in each periodic section of the antidiagonal will be either 0 or 1 and each entry represents parts totaling $2^{k-1}$, so if nonzero the amount contributed by those entries to the partition will be $2^{r(k)} 2^{k-1}$.

The end result will be a partition with odd orbit size.  We chose parts in the matrices $M_j$, $j=2n-1$.  We can freely choose $i$ of the $t$ available periods, giving $\binom{t}{i}$ possible fillings for each antidiagonal.  We end up with the generating function

\begin{multline*}\prod_{n=1}^\infty \prod_{k=1}^\infty \sum_{i=0}^t \binom{t}{i} q^{(2n-1)2^{k-1} 2^{j(k)} i} = \prod_{n=1}^\infty \prod_{k=1}^\infty (1+q^{(2n-1)2^{k-1} 2^{j(k)}})^t \\
\equiv \prod_{n=1}^\infty \prod_{k=1}^\infty (1+q^{(2n-1)2^{k-1} })^k \pmod{2}.
\end{multline*}

In the last line we used the fact that $(1-q^k)^2 \equiv (1-q^{2k}) \pmod{2}$.  Alas, this is easily seen to be equivalent to partitions into distinct odd parts.  If the experiment is run for general modulus, we can fairly easily obtain the fact that

\begin{theorem}Fix $b > 1$.  Letting $k_m = r+1$ if $b^r \| m$, $$\sum_{n=0}^\infty p(n) q^n \equiv \prod_{m=1}^\infty (1+q^m+q^{2m}+ \cdots + q^{(b-1)m})^{k_m} \pmod{b}.$$  That is, the number of partitions of $n$ is equivalent modulo $b$ to the number of partitions of $n$ in which parts of size $m$ appear at most $b-1$ times in each of $k_m$ types.
\end{theorem}

\phantom{.}

\noindent \textbf{Remark 4:} If we wish to consider parts appearing at most once, and exponents only 0, 1, or 2, we may consider the fact that $\frac{1}{(q;q)_\infty} = \prod_{n=1}^\infty (1+q^n)^{a_n}$ in which $a_n$ is the power of 2 that divides $n$, plus 1.  Reducing products modulo 3 we obtain $$\sum_{n=1}^\infty p(n) q^n \equiv \prod_{n=1}^\infty (1+q^n)^{a_n} \pmod{3}$$ \noindent in which, if $n = 2^r 3^k m$, $gcd(m,6)=1$, then $a_n$ is the residue modulo 3 of the $k$-th iteration of the map starting from 0th iteration $r+1$ and, at each successive iteration, adds the floor of the previous exponent, divided by 3, to the residue of the previous exponent mod 3. 

\section{Generalization of a theorem related to third-order mock theta functions}\label{MockSection}

We next illustrate the utility of part-frequency matrices for proving combinatorial identities of suitable type.  In this section and the succeeding we employ the standard notation $$(a;q)_n = (1-a)(1-aq)\dots (1-aq^{n-1}) \quad , \quad (a;q)_\infty = \prod_{k=0}^\infty (1-aq^k).$$

A classical collection of objects in partition theory is Ramanujan's mock theta functions, which have spurred a phenomenal amount of work \cite{MockTheta}.

In \cite{ADY} Andrews, Dixit and Yee consider combinatorial interpretations of $q$-series identities related to several of the third order mock theta functions.  Most of these involve partitions into odd or distinct parts with restrictions on the sizes of odd or even parts appearing.  One such identity is (Theorem 3.4 in \cite{ADY})

\begin{theorem}\label{34Orig} $$\sum_{n=1}^\infty \frac{q^n}{(q^{n+1};q)_n (q^{2n+1};q^2)_\infty} = -1 + (-q;q)_\infty$$
\end{theorem}

\noindent which, interpreted combinatorially, yields (Theorem 3.5 in \cite{ADY})

\begin{corollary} The number of partitions of a positive integer $n$ with unique smallest part in which each even part does not exceed twice the smallest part equals the number of partitions of $n$ into distinct parts.
\end{corollary}

Using part-frequency matrices, we can re-prove and generalize this to

\begin{theorem} The number of partitions of $n$ with smallest part appearing fewer than $m$ times, in which each part divisible by $m$ does not exceed $m$ times the smallest part, equals the number of partitions of $n$ into parts appearing fewer than $m$ times.
\end{theorem}

\begin{proof}
Let the smallest part of the partition be $j m^k$, and consider the matrices $M_i$ with modulus $m$.  For any other matrix $M_i$ with $m \nmid i$, if $i < j m^k$, then there is exactly one value of $k_1$ for which $j m^k < i m^{k_1} < j m^{k+1}$.  If $i > j m^k$, then only parts $i m^0$ are allowed under the conditions of the theorem.  Thus, in each matrix $M_i$, there is one and exactly one part size $i m^{k_1}$ permissible, with the exception of $M_j$, in which $j m^k$ appears with nonzero entry $a_{0,jm^k}$ and all other $a_{b,jm^k} = 0$, along with possible entries $a_{b,jm^{k+1}}$.

Now the map is simply a revised form of Glaisher's bijection: we transpose the portions of the matrices $M_i$ starting at row $k_1$, i.e, we exchange the entries in row $k_1$ for entries in the $m^0$ column starting no lower than row $k_1$.

The resulting matrix still has smallest part $j m^k$, and any part larger than this may appear fewer than $m$ times.  Observe that as long as $jm^k$ remains the smallest part, the allowable parts for a partition of the previous type are still identifiable.  In particular, we can take a partition into parts appearing fewer than $m$ times and exchange the $m^0$ column of matrix $M_i$ for the $k_1$ row which makes $j m^k < i m^{k_1} < j m^{k+1}$, or for the 0 row if $i > j m^k$.
\end{proof}

\noindent \textbf{Example:} Let $m=2$ for convenience and say $\lambda = (15,15,15,10,10,8,6)$.  Write the matrices $M_j$ and mark the allowed row:

\begin{center} \begin{tabular}{ccccc} $\begin{array}{c|ccc} 1 & & & \\ \hline & 0 & 0 & 0 \\ & 0 & 0 & 0 \\ & 0 & 0 & 0 \\ * & 1 & 0 & 0 \end{array}$ & $\begin{array}{c|ccc} 3 & & & \\ \hline & 0 & 0 & 0 \\ & 1 & 0 & 0 \\ * & 0 & 0 & 0 \\ & 0 & 0 & 0 \end{array}$ & $\begin{array}{c|ccc} 5 & & & \\ \hline & 0 & 0 & 0 \\ * & 0 & 1 & 0 \\ & 0 & 0 & 0 \\ & 0 & 0 & 0\end{array}$ & $\dots$ & $\begin{array}{c|ccc} 15 & & & \\ \hline * & 1 & 1 & 0 \\ & 0 & 0 & 0 \\ & 0 & 0 & 0 \\ & 0 & 0 & 0 \end{array}$ \end{tabular} \end{center}

Transpose the part of the matrices at and below the marked row.

\begin{center} \begin{tabular}{ccccc} $\begin{array}{c|ccc} 1 & & & \\ \hline & 0 & 0 & 0 \\ & 0 & 0 & 0 \\ & 0 & 0 & 0 \\ * & 1 & 0 & 0 \end{array}$ & $\begin{array}{c|ccc} 3 & & & \\ \hline & 0 & 0 & 0 \\ & 1 & 0 & 0 \\ * & 0 & 0 & 0 \\ & 0 & 0 & 0 \end{array}$ & $\begin{array}{c|ccc} 5 & & & \\ \hline & 0 & 0 & 0 \\ * & 0 & 0 & 0 \\ & 1 & 0 & 0 \\ & 0 & 0 & 0\end{array}$ & $\dots$ & $\begin{array}{c|ccc} 15 & & & \\ \hline * & 1 & 0 & 0 \\ & 1 & 0 & 0 \\ & 0 & 0 & 0 \\ & 0 & 0 & 0 \end{array}$ \end{tabular} \end{center}

\phantom{.}

We have the refined identity
$$\frac{(q^{mn};q^m)_\infty}{(q^{n+1};q)_{(m-1)n} (q^{mn};q)_\infty} = \frac{(q^{m(n+1)};q^m)_\infty}{(q^{n+1};q)_\infty}$$

\noindent which, multiplied on both sides by $\frac{q^n(1-q^{(m-1)n})}{1-q^n}$ and summed over all $n$ save the exceptional case $n=0$, gives the additionally parametrized $q$-series identity

\begin{corollary}
$$\sum_{n=1}^\infty \frac{q^n(1-q^{(m-1)n})(q^{mn};q^m)_\infty}{(1-q^n)(q^{n+1};q)_{(m-1)n} (q^{mn};q)_\infty} = -1 + \frac{(q^m;q^m)_\infty}{(q;q)_\infty}$$
\end{corollary}

\noindent for which $m=2$ is the original identity (\ref{34Orig}).

Using the part-frequency matrices we can play with this theorem in a variety of ways.  For instance, suppose we relax the restriction that the smallest part be unique to require that it be not divisible by $m$.  Then we can pivot the row in matrix $M_j$ of smallest part $jm^k$ around the $a_{1,jm^k}$ entry, preserving the smallest part, and any entries of size $a_{i,jm^{k+1}}$ are rotated to $a_{0,jm^{k+1+i}}$.  Now the resulting partition has parts which appear fewer than $m$ times, except that multiples of the smallest part by $m^k$ may appear less than $m^2$ times.  Other variations can be explored at leisure.

\phantom{.}

\noindent \textbf{Question 4.} Can other known theorems related to third-order mock theta functions be generalized in this fashion, or the behavior of the relevant partitions related to Glaisher-type maps?

\end{document}